\documentclass[times,doublespace,letterpaper,12pt]{extarticle}
\usepackage[utf8]{inputenc}
\usepackage[english]{babel}
\usepackage[top=3cm, outer=3cm, bottom=3cm, inner=3cm]{geometry}
\usepackage{amssymb,amsmath,amsthm}
\usepackage{mathtools}

\newtheorem{lemma}{Lemma}
\newtheorem{corollary}{Corollary}
\newtheorem{theorem}{Theorem}
\newtheorem{proposition}{Proposition}

\theoremstyle{definition}

\newtheorem{remark}{Remark}

\DeclareMathOperator{\Diff}{Diff}
\DeclareMathOperator{\ad}{ad}
\DeclareMathOperator{\Char}{char}
\DeclareMathOperator{\End}{End}
\DeclareMathOperator{\sL}{sl}
\DeclareMathOperator{\Ker}{Ker}
\DeclareMathOperator{\St}{St}
\newcommand{\myeq}{\stackrel{\mathclap{\normalfont\mbox{def}}}{=}}

\begin{document}

\title{Prime Lie algebras satisfying the standard Lie identity of degree $5$}

\begin{center}
{\Large\bf Prime Lie algebras satisfying the standard Lie identity of degree~$5$}

\smallskip

Gleb Pogudin\footnote{Institute for Algebra, Johannes Kepler University, Linz, Austria, e-mail: pogudin.gleb@gmail.com},\\
Yuri P. Razmyslov\footnote{Department of Mechanics and Mathematics, Moscow State University}

\end{center}

\begin{abstract}
	For every commutative differential algebra one can define the Lie algebra of special derivations.
    It is known for years that not every Lie algebra can be embedded to the Lie algebra of special derivations of some differential algebra.
    More precisely, the Lie algebra of special derivations of a commutative algebra always satisfy the standard Lie identity of degree $5$.
    The problem of existence of such embedding is a long-standing problem (see \cite{YuPSt5,PoinsotAAM,Poinsot}), which is closely related to the Lie algebra of vector fields on the affine line (see \cite{YuPSt5}).
    It was solved by Razmyslov in \cite{YuPSt5} for simple Lie algebras satisfying this identity (see also \cite[Th. 16]{Poinsot}).
    We extend this result to prime (and semiprime) Lie algebras over a field of zero characteristic satisfying the standard Lie identity of degree~$5$.
    As an application, we prove that for any semiprime Lie algebra the standard identity $St_5$ implies all other identities of the Lie algebra of polynomial vector fields on the affine line.
    
    We also generalize some previous results about primeness of the Lie algebra of special derivations of a prime differential algebra to the case of non-unitary differential algebra. 


\end{abstract}

\section{Preliminaries}

    Throughout the paper all fields are assumed to be of characteristic zero.
	Let $R$ be a ring.
	A map $D\colon R \to R$ satisfying $D(a + b) = D(a) + D(b)$ and $D(ab) = aD(b) + D(a)b$ for all $a, b \in R$ is called a \textit{derivation}.
    A \textit{differential ring} $R$ is a ring with a specified derivation. 
	In this case we will denote $D(x)$ by $x^{\prime}$ and $D^n(x)$ by $x^{(n)}$.
	A differential ring which is a field will be called a \textit{differential field}.
    For a nonzero $s \in A$ by $A_s$ we denote the localization of $A$ with respect to $s$.
    Furthermore, if $A$ is a differential ring, the derivation on $A$ can be uniquely extended to $A_s$.
    
    A differential ring $A$ is said to be a \textit{differential $k$-algebra} over a field $k$ if $A$ is a $k$-algebra and the derivation is $k$-linear.
	An ideal $I$ of a differential ring $R$ is said to be a \textit{differential ideal}, if $a^{\prime} \in I$ for all $a \in I$.
	The differential ideal generated by $a_1, \ldots, a_n \in I$ will be denoted by $[a_1, \ldots, a_n]$.
    A differential ring $A$ is \textit{prime} if $I \cdot J \neq 0$ for any nonzero differential ideals $I$ and $J$.
	A differential ring $A$ is \textit{semiprime} if $I^2 \neq 0$ for any nonzero differential ideal $I$.
    Let us remark that the definition above is a special case of the general notion of a prime algebra of an arbitrary signature (see \cite[\S 3]{YuPmonogr}), where a differential algebra is considered as an algebra with binary multiplication and unary derivation.

	Let $A$ be a differential $k$-algebra.
    We consider the polynomial ring $A[x, x^{\prime}, \ldots, x^{(n)}, \ldots]$, where $x, x^{\prime}, x^{\prime\prime}, \ldots$ are algebraically independent variables.
    Extending the derivation from $A$ to $A[x, x^{\prime}, \ldots]$ by $D(x^{(n)}) = x^{(n + 1)}$ we obtain a differential algebra.
    This algebra is called the \textit{algebra of differential polynomials} in $x$ over $A$ and we denote it by $A\{ x\}$.
    Iterating this construction, we define the algebra of differential polynomials in variables $x_1, \ldots, x_n$ over $A$ and denote it by $A\{x_1, \ldots, x_n\}$.

	A Lie algebra $L$ is said to be \textit{prime} if $[A, B] \neq 0$ for any nonzero ideals $A, B \subset L$. 
    A Lie algebra $L$ is said to be \textit{semiprime} if $[A, A] \neq 0$ for any nonzero ideal $A \subset L$. 
    
%%%%%%%%%%%%%%%%%%%%%%%%%%%%%%%%%%%%%%%%%%%%%%%%%%%%%%%%%%%%%%%%%%%%%%%%%%%%%%%%%%%%%%%%%%%%%%%%%%%%%%%%%%

\section{On the embedding of a Lie algebra in the Lie algebra of special derivations}
    For a differential $k$-algebra $A$, let us define \textit{the Lie algebra of special derivations} $\Diff A = \{ a\partial | a\in A \}$. 
    Lie bracket on $\Diff A$ is defined by:
    $$
  	[a\partial, b\partial] = (ab^{\prime} - a^{\prime}b) \partial
   	$$
	$\Diff A$ can be thought as an algebra of homogeneous linear differential operators over $A$. 
    See also \cite[\S 46]{YuPmonogr}.
    
    Not all Lie algebras can be represented as a subalgebra of the Lie algebra of special derivations of some differential algebra.
	It is well-known (see~\cite{YuPSt5}) that:
    
    \begin{proposition}\label{prop:St5}
		Let $A$ be a differential $k$-algebra.
        Then, algebra $\Diff A$ satisfies the standard Lie identity of degree $5$ (shortly, $\St_5$):
        
        $$
    	\sum\limits_{\sigma \in S_4} (-1)^{\sigma} \ad x_{\sigma(1)} \ad x_{\sigma(2)} \ad x_{\sigma(3)} \ad x_{\sigma(4)} z,
    	$$

		where $(-1)^{\sigma}$ is a sign of permutation $\sigma$.
	\end{proposition}

	In \cite{YuPSt5} Yu.~Razmyslov proved that in the case of simple Lie algebra the converse holds.
    Precisely, he proved that every simple Lie algebra over a field with $\Char k \neq 2$ satisfying the standard Lie identity of degree $5$ can be embedded to the Lie algebra of special derivations of some differential algebra. 
    Using this embedding, Razmyslov proved that for any simple Lie algebra the standard identity $St_5$ implies all other identities of the Lie algebra of polynomial vector fields on the affine line.
    In this paper we extend this result to semiprime Lie algebras (see Corolary~\ref{cor:identities}).
    
    In his papers \cite{Poinsot,PoinsotAAM}, L.~Poinsot used a language of category theory to describe the correspondence between differential algebra and its Lie algebra of special derivations. 
    It turns out that the correspondence is functorial.
    This functor has left adjoint, so for every Lie algebra $L$ we assign a differential algebra.
    Poinsot calls the Lie algebra of special derivations of this differential algebra \emph{the Wronskian envelope of $L$}, and Theorem~9 in \cite{Poinsot} (see also Proposition~49 in \cite{PoinsotAAM}) describes the Wronskian envelope of $L$ in terms of generators and relations.
    Proposition~\ref{prop:St5} shows that the canonical homomorphism from $L$ to its Wronskian envelope need not be injective (see also Example~50 in \cite{PoinsotAAM}).
    If it turns out to be injective, then $L$ is said to be \emph{Wronskian special} (see \cite[\S 4]{Poinsot}).
        
    Using result of Yu.~Razmyslov, L.~Poinsot proved (see \cite[Th. 16]{Poinsot}) that subalgebras of the direct product, where each component is either simple and satisfies $\St_5$ or is abelian, are Wronskian special.
    Our goal is to prove that every semiprime Lie algebra satisfying the standard Lie identity of degree $5$ is Wronskian special (see Corollary~\ref{cor:semiprime}).
    
%     Consider the algebra of formal power series $k[[x]]$ with respect to derivation $\frac{\partial}{\partial x}$.
%     We will denote the corresponding algebra of special derivations by $\widetilde{W}_1$.

	We will use the following construction from \cite{YuPSt5}. 
    Let us consider a Lie algebra $L$ satisfying the standard Lie identity of degree $5$.
    By $\langle g_1, g_2, g_3 \rangle$ we denote the element 
    $$\langle g_1, g_2, g_3 \rangle \myeq \sum\limits_{\sigma \in S_3} (-1)^{\sigma} \ad g_{\sigma(1)} \ad g_{\sigma(2)} \ad g_{\sigma(3)} \in \End_k L,$$ 
    where $g_1, g_2, g_3 \in L$.
    Let $R(L)$ be the associative subalgebra of $\End_k L$ generated by elements of the form $\langle g_1, g_2, g_3 \rangle$.
    In a paper \cite{YuPSt5} in was shown that $R(L)$ is commutative, and for all $g \in L$ the formula:
    $$
    {}^{g}\langle g_1, g_2, g_3\rangle \myeq [\ad g, \langle g_1, g_2, g_3 \rangle] = \langle [g, g_1], g_2, g_3 \rangle + \langle g_1, [g, g_2], g_3 \rangle + \langle g_1, g_2, [g, g_3] \rangle
    $$ 
	defines a derivation on $R(L)$.

	\begin{proposition}[\cite{YuPmonogr}, Proposition~46.3]\label{prop:identities}
		Let $L$ be a Lie algebra and $L$ satisfies the standard Lie identity of degree $5$.
        Then, for all $a, b \in R(L)$ and $f, g \in L$:
        \begin{enumerate}
			\item ${}^g a h = {}^h a g$;
            
            \item ${}^{bg} a = b \left({}^g a\right)$.
		\end{enumerate}
	\end{proposition}
    
    The next section contains the proof of the following result.
    
    \begin{theorem}\label{th:lie_to_diff}
		Let $L$ be a prime Lie algebra over a field $k$ of zero characteristic satisfying the standard Lie identity of degree~$5$.
        Then, $L$ can be embedded to the Lie algebra of special derivations of some prime differential algebra.
	\end{theorem}

	\begin{corollary}\label{cor:prime}
		Let $L$ be a prime Lie algebra over a field $k$ of zero characteristic satisfying the standard Lie identity of degree~$5$.
        Then, $L$ is Wronskian special.
	\end{corollary}
    
    \begin{corollary}\label{cor:semiprime}
    	Let $L$ be a semiprime Lie algebra over a field $k$ of zero characteristic satisfying the standard Lie identity of degree~$5$.
        Then, $L$ can be embedded to the Lie algebra of special derivations of some semiprime differential algebra.
        
        In particular, $L$ is Wronskian special.
	\end{corollary}

	\begin{proof}
		Due to Theorem~1.1 from \cite{Amitsur}, zero ideal in $L$ can be represented as an intersection~$\bigcap\limits_{\alpha\in\Lambda} I_{\alpha}$ of ideals, where $L / I_\alpha$ is a prime Lie algebra for all $\alpha \in \Lambda$.
        Then, there is an embedding $\pi \colon L \to \prod\limits_{\alpha \in \Lambda} L / I_{\alpha}$.
        Since $L / I_{\alpha}$  satisfies the standard Lie identity of degree~$5$ for all $\alpha \in \Lambda$, Theorem~\ref{th:lie_to_diff} implies that for every $\alpha$ there exists an embedding $i_\alpha \colon L / I_{\alpha} \to \Diff A_{\alpha}$, where $A_{\alpha}$ is a prime differential algebra.
        Taking the composition of these maps, we obtain an embedding 
        $$L \to \prod\limits_{\alpha \in \Lambda} \Diff A_{\alpha} \cong \Diff \left(\prod\limits_{\alpha \in\Lambda} A_{\alpha}\right).$$
        The algebra $\prod\limits_{\alpha \in \Lambda} A_{\alpha}$ is a semiprime differential algebra, so we are done.
	\end{proof}

	The following proposition is an analogue of the implication $4 \Rightarrow 2$ in the Main Theorem of \cite{YuPSt5}.
    
    \begin{proposition}\label{prop:wronsk_special}
    	Let $A$ be a differential algebra over a field $k$ of zero characteristic.
        Then, $\Diff A$ satisfies all identites satisfied by the algebra $\widetilde{W}_1 = \Diff k[[x]]$.

		In particular, if $L$ is a Wronskian special Lie algebra over a field $k$ of zero characteristic, then $L$ satisfies all identites satisfied by the algebra $\widetilde{W}_1 = \Diff k[[x]]$.
	\end{proposition}
    
    \begin{proof}
		Assume the contrary.
        Let $f(x_1, \ldots, x_n)$ be a Lie polynomial such that $f(x_1, \ldots, x_n) = 0$ is an identity in $\widetilde{W}_1$, but there exist $a_1, \ldots, a_n \in A$ such that $f(a_1\partial, \ldots, a_n\partial) \neq 0$.
        Consider the algebra of differential polynomials $B = k\{ y_1, \ldots, y_n \}$.
        There exists a differential polynomial $p(y_1, \ldots, y_n)$ such that $f(y_1\partial, \ldots, y_n\partial) = p(y_1, \ldots, y_n)\partial$.
        Since $p(a_1, \ldots, a_n) \neq 0$, $p$ is a nonzero differential polynomial.

		$B$ is isomorphic, as a $k$-algebra, to the polynomial algebra in $y_1, \ldots, y_n$ and their derivatives, so there exists a homomorphism (not necessarily differential) $\varphi\colon B \to k$ such that $\varphi\left( p(y_1, \ldots, y_n) \right) \neq 0$.
        It is easy to verify that over a field of zero characteristic formula
        $$
        \widetilde{\varphi}(a) = \sum\limits_{j = 0}^{\infty} \varphi\left( a^{(n)} \right) \frac{x^j}{j!}
        $$
        defines a homomorphism of differential algebras from $B$ to $k[[x]]$ (for details, see \cite[\S 44.3]{YuPmonogr}).
        
        The homomorphism $\widetilde{\varphi}$ gives rise to a homomorphism of the algebras of special derivations $\widehat{\varphi}\colon \Diff B \to \widetilde{W}_1$.
        Moreover, $\varphi\left( p(y_1, \ldots, y_n) \right) \neq 0$ implies that $\widetilde{\varphi}\left( p(y_1, \ldots, y_n) \right) \neq 0$.
        Then
        $$
        f\left( \widehat{\varphi}(y_1), \ldots, \widehat{\varphi}(y_1) \right) = \widehat{\varphi}\left( f(y_1\partial, \ldots, y_n\partial) \right) = \widehat{\varphi}\left( p(y_1, \ldots, y_n)\partial \right) = \widetilde{\varphi} \left( p(y_1, \ldots, y_n)\right) \partial \neq 0
        $$
        This, $f(x_1, \ldots, x_n)$ is not an identity in $\widetilde{W}_1$, so we arrived at a contradiction.
	\end{proof}
    
    \begin{remark}
		It is possible to remove the assumtion of zero characteristic from Proposition~\ref{prop:wronsk_special} by using the algebra of divided power series instead of $k[[x]]$, see \cite[\S 44.3]{YuPmonogr} for details.
	\end{remark}
        
    \begin{corollary}\label{cor:identities}
		Let $L$ be a semiprime Lie algebra over a field $k$ of zero characteristic satisfying the standard Lie identity of degree~$5$.
        Then, $L$ satisfies all identities satisfied by the algebra $\widetilde{W}_1 = \Diff k[[x]]$.
	\end{corollary}

    \begin{remark}
		The Lie algebra $\Diff k[[x]]$ is a completion of the Lie algebra $\Diff k[x]$ of the polynomial vector fields on the affine line, so they satisfy the same identities.
        Thus, we proved that for any semiprime Lie algebra over a field of zero characteristic the standard identity $St_5$ implies all other identities of the Lie algebra of polynomial vector fields on the affine line.
	\end{remark}

%%%%%%%%%%%%%%%%%%%%%%%%%%%%%%%

\section{Proof of Theorem~\ref{th:lie_to_diff}}

 	First of all, we embed $L$ to the algebra of special derivations of some (not necessarily prime) differential algebra.
    Let $\overline{L}$ be the free algebra in a variety generated by $L$ with one free generator $\mathbf{g}$ over $L$ (see \cite[\S 10]{burris}).
    More precisely, let $L\langle \mathbf{g} \rangle$ be the free Lie algebra of rank one generated by $\mathbf{g}$ over $L$.
    By $T$ denote the ideal
    $$
    T = \{ F(\mathbf{g}) \in L\langle \mathbf{g} \rangle \mid F(a) = 0 \mbox{ for all } a \in L \}
    $$
    Then, $\overline{L} = L\langle \mathbf{g} \rangle / T$.
    Theorem 10.12 from \cite{burris} implies that $\overline{L}$ satisfies all identites of $L$.
       
    By $\overline{I}$ denote the ideal in $\overline{L}$ generated by elements of the form ${}^{\mathbf{g}} \langle x, y, z \rangle$ for all $x, y, z \in L$,
    and by $I$ denote the ideal in $L$ generated by elements of the form ${}^{g} \langle x, y, z \rangle$ for all $x, y, z, g \in L$.
   
    \underline{Assume that $\overline{I} = 0$}. Then, $L$ satisfies the identity ${}^{g} \langle x, y, z \rangle = 0$.
    If $R(L) = 0 $, then the identity $\langle x, y, z\rangle = 0$ holds in $L$.
    Due to Lemma~39.4 from \cite{YuPmonogr} in this case $L$ is metabelian, i.e. $\left[ [L, L], [L, L]\right] = 0$.
    Then, $L$ is not prime, so $R(L) \neq 0$.
    Moreover, there are no zero divisors in $R(L)$.
    Indeed, let $a, b \in R(L)$ be nonzero elements such that $ab = 0$.
    For arbitrary $f, g \in L$ we have $[f, ag] = {}^fa + a[f, g] = a[f, g]$, so $aL$ and $bL$ are nonzero Lie ideals in $L$.
    Let $ag \in aL$ and $bf \in bL$, then
    $$
    [ag, bf] = b[ag, f] = ab[f, g] = 0
    $$
    So, $\left[ aL, bL \right] = 0$, and $L$ is not prime.
    Hence, $R(L)$ is an integral domain.
    By $Q$ denote the field of fractions of $R(L)$.
    Fix $g_1, g_2, g_3 \in L$ such that $\langle g_1, g_2, g_3 \rangle \neq 0$.
    Then, the identity 
    $$g = \frac{\langle g_1, g_2, g_3 \rangle}{\langle g_1, g_2, g_3 \rangle} g = \frac{1}{\langle g_1, g_2, g_3 \rangle} \left( \langle g, g_2, g_3 \rangle g_1 + \langle g_1, g, g_3 \rangle g_2 + \langle g_1, g_2, g\rangle g_3 \right)$$ 
    implies that $QL$ is a three-dimensional Lie algebra over $Q$.
    Let $K$ be the algebraic closure of $Q$.
    Over the algebraically closed field $K$ any three-dimensional nonsolvable Lie algebra is isomorphic to $\sL_2(K)$ (see \cite[\S 1.4]{jacobson}).
    Thus, there is an embedding of $L$ to $\sL_2(K)$.
    The algebra $\sL_2(K)$ is isomorphic to the subalgebra $\langle \partial, x \partial, x^2\partial \rangle \subset K[x]\partial$ of the Lie algebra of special derivations of $K[x]$.
    Hence, we constructed an embedding  $L \to \Diff K[x]$.
   
   	\underline{Thus, assume $\overline{I} \neq 0$.}
    By $\overline{I}^{\ast}$ denote $\overline{I} \backslash 0$.
    Consider the Lie algebra $$\widetilde{L} = \prod\limits_{a \in \overline{I}^{\ast}} \overline{L}.$$
    Lie algebra $\widetilde{L}$ is a left module over the algebra $\widetilde{R} = \prod\limits_{a \in \overline{I}^{\ast}} R(\overline{L})$.
    By $s$ denote the element of $\widetilde{R}$ such that the coordinate of $s$ corresponding to $a \in \overline{I}^{\ast}$ equals $a$ for all $a \in \overline{I}^{\ast}$.
    Consider the localization of $\widetilde{R}$ and $\widetilde{R}$-module $\widetilde{L}$ with respect to $s$.
    The localization $\widetilde{L}_s$ has a natural Lie algebra structure defined by the formula ($f, g \in \widetilde{L}$):
    $$
    \left[ \frac{f}{s^m}, \frac{g}{s^k} \right] = -k\frac{ {}^fs }{ s^{m + k + 1} }g + m\frac{ {}^gs }{ s^{m + k + 1} }f + \frac{1}{s^{m + k}}[f, g]
    $$
   
    \begin{lemma}\label{lem:spec_diff}
		Consider $\widetilde{R}_s$ as a differential algebra with respect to component-wise action of $\ad {\bf g}$.
        Then, $\Diff \widetilde{R}_s = \widetilde{L}_s$.
	\end{lemma}

	\begin{proof}
		The inclusion $\Diff \widetilde{R}_s \subset \widetilde{L}_s$ is obvious.
        Let $g \in \widetilde{L}_s$, and by $g_a$ we denote the coordinate of $g$ corresponding to an element $a = \sum c_i {}^{\bf g}\langle x_i, y_i, z_i \rangle \in \overline{I}^{\ast}$.
        Then (the last equality is due to the first statement of Proposition~\ref{prop:identities}):
        $$
        g_a = a^{-1} a g_a = a^{-1} \sum c_i {}^{\bf g}\langle x_i, y_i, z_i \rangle g_a = a^{-1} \sum c_i {}^{g_a}\langle x_i, y_i, z_i \rangle {\bf g}
        $$
        By $b$ denote the element of $\tilde{R}_s$ such that $b_a = \sum c_i {}^{g_a}\langle x_i, y_i, z_i \rangle$.
        Then, $g = s^{-1} b {\bf g}$, so $g \in \Diff \widetilde{R}_s$.
	\end{proof}
	
    Let us return to the proof of Theorem~\ref{th:lie_to_diff}.
    There is a natural injection from $L$ to $\widetilde{L}$ sending an element $g \in L$ to an element $\tilde{g} \in \widetilde{L}$ whose coordinates all are equal to $g$.
    Composing the injection with the localization map we obtain the homomorphism $\varphi\colon L \to \widetilde{L}_s = \Diff \widetilde{R}_s$.
    If $\Ker\varphi = 0$, then Lemma \ref{lem:spec_diff} gives us a desired injection to the Lie algebra of special derivations.
    
    Assume, on the contrary, that $\Ker\varphi \neq 0$.
    $\Ker\varphi$ consists of $f \in L$ such that there exists $n$ such that for all $a \in\overline{I}$ equality $a^n f = 0$ holds.
    Equivalently, for all $a \in I$ equality $a^n f = 0$ holds.
    Since $k$ is a field of zero characteristic, we can linearize $a^n f = 0$ with respect to $a$ (see \cite[Th. 1.3.8]{zaicev}).
    Then, we obtain that for all $a_1, \ldots, a_n \in I$ equality $a_1\ldots a_n f = 0$ holds.
    Let $f$ be a nonzero element of $\Ker\varphi$.
    Let $n$ be the minimal natural number such that for all $a_1, \ldots, a_n \in I$ equality $a_1\ldots a_n f = 0$ holds.
    Then, there exist $a_2, \ldots, a_n \in I$ such that $a_2\ldots a_n f \neq 0$.
    Hence, $a(a_2\ldots a_n)f = 0$ for all $a \in I$.
    Thus, the set $A = \{f \in L \mbox{ such that }If = 0\}$ is not equal to $\{ 0\}$.
    
    \begin{lemma}
	\begin{enumerate}
		\item $[\ad g, I] \subset I$ for all $g \in L$;
        \item $A$ is an ideal;
        \item $B = IL$ is an ideal.
	\end{enumerate}
	\end{lemma}
    
    \begin{proof}
	\begin{enumerate}
		\item An element $a \in I$ can be represented as $\sum c_i {}^{g_i}\langle x_i, y_i, z_i \rangle$, where $c_i \in R(L)$ and $x_i, y_i, z_i, g_i \in L$.
        Then:
        \begin{multline*}
        [\ad g, a] = \sum ( [g, c_i] {}^{g_i}\langle x_i, y_i, z_i\rangle + c_i {}^{[g, g_i]}\langle x_i, y_i, z_i \rangle + \\ + c_i {}^{g_i}\langle [g, x_i], y_i, z_i\rangle  + c_i {}^{g_i}\langle x_i, [g, y_i], z_i\rangle + {}^{g_i}\langle x_i, y_i, [g, z_i]\rangle )
        \end{multline*}
        
        \item Let $g\in L$, $f \in A$, $a \in I$.
        Then 
        $$0 = [g, af] = [\ad g, a]f + a[f, g]$$
        Since $[\ad g, a] \in I$, the first summand equals zero.
        Then, the second equals zero, too.
        
        \item Let $f, g \in L$, $a \in I$.
        Then, $[g, af] = [\ad g, a]f + a[f, g]$.
        Since both $[\ad g, a]f$ and $a[f, g]$ lie in $B$, then $[g, af] \in B$.
	\end{enumerate}
	\end{proof}
    
    We claim that $[A, B] = 0$.
    Indeed, let $g \in L$, $f \in A$ and $a \in I$
    $$[f, ag] = {}^f a g + a[f, g] = {}^g a f + a[f, g] = 0$$
    But we assumed that both ideals are non zero, so we arrived at a contradiction.
        
    Thus, we obtained the injection $\varphi\colon L \to \Diff C$, where $C = \widetilde{R}_s$.
    For a subset $X \subset L$, we denote by $\mu(X)$ the set $\{a \in C | a\partial \in \varphi(X)\}$.
    Denote by $\mathfrak{N}$ the set of differential ideals $J \subset C$ such that $J \cap \mu(L) = 0$.
    $\mathfrak{N}$ is a poset with respect to inclusion and satisfies conditions of Zorn's lemma.
    Let $J$ be a maximal element of $\mathfrak{N}$.
    It follows from the definition of $\mathfrak{N}$ that the composition $L \to \Diff C \to \Diff (C / J)$ is an injection.

	It suffices to show that $C / J$ is prime.
    Let $J_1, J_2$ be ideals in $C / J$ such that $J_1J_2 = 0$.
    Then, $0 = [J_1\partial, J_2\partial] \supset [L \cap J_1\partial, L\cap J_2\partial]$.
    Since $L$ is prime, either $\mu(L) \cap J_1 = 0$ or $\mu(L) \cap J_2 = 0$.
    Maximality of $J$ implies that either $J_1 = 0$ or $J_2 = 0$.
    Thus, $C / J$ is prime.

%%%%%%%%%%%%%%%%%%%%%%%%%%%%

\section{Primeness of the Lie algebra of special derivations of a prime differential algebra}

	In the previous section we constructed the embedding of a prime Lie algebra to the Lie algebra of special derivations of some prime differential algebra.
    It is natural to ask whether the Lie algebra of special derivations of a prime differential algebra is a prime Lie algebra.
    The question was extensively studied in literature (see \cite{Jordan1,DeltaPrimeNot2,DeltaPrime2} and references there), and the answer is affirmative in most cases.
    However, all mentioned results are applicable only to differential algebras with identity.
    But we can construct the Lie algebra of special derivations from a prime differential algebra without identity, so it is interesting to extend the results mentioned above to the non-unitary case.

	The main result of this section is the following theorem.

    \begin{theorem}\label{th:diff_to_lie}
		Let $A$ be a prime differential algebra over a field $k$ of characteristic zero.
        Then, $\Diff A$ is prime.
	\end{theorem}

	\begin{proof}
	We start with two auxiliary results that can be of independent interest.

	\begin{proposition}\label{prop:prime_characterization}
		A differential algebra $A$ is prime (resp. semiprime) if and only if for all $a, b \in A$ (resp. for all $a \in A$) there exists a nonegative integer $n$ such that $a^{(n)} b \neq 0$ (resp. $a^{(n)}a \neq 0$).
	\end{proposition}
    
    \begin{proof}
		Assume that for all $a, b \in A$ (resp. for all $a \in A$) there exists a nonegative integer $n$ such that $a^{(n)} b \neq 0$ (resp. $a^{(n)}a \neq 0$).
        Let $I, J \subset A$ be nonzero differential ideals (resp. $I \subset A$ is a nonzero differential ideal).
        Pick two nonzero element $a \in I$ and $b \in J$ (resp. $a, b \in I$).
        There exists $n$ such that $a^{(n)}b \neq 0$ (resp. $a^{(n)} a \neq 0$).
        Since $a^{(n)} \in I$, $a^{(n)}b \in IJ$ (resp. $a^{(n)}a \in I^2$).
        
        Now, let $A$ be a prime (resp. semiprime) differential algebra. 
        Since the product of nonzero principial ideals generated by $a$ and $b$ (resp. square of the principal ideal generated by $a$) is nonzero, there exist nonegative integers $i$ and $j$ such that $a^{(i)} b^{(j)} \neq 0$ (resp. $a^{(i)} a^{(j)} \neq 0$).
        Let sum $i + j$ be the minimal possible.
        Choose the minimal $j$ among all such pairs with fixed $i + j$.
        
        Assume that $j > 0$.
        Then, $a^{(i)} b^{(j - 1)} = 0$ (resp. $a^{(i)} a^{(j - 1)} = 0$).
        Hence, 
        $$\left( a^{(i)} b^{(j - 1)}\right)^{\prime} = a^{(i + 1)}b^{(j - 1)} + a^{(i)}b^{(j)} = 0 $$
        $$\left(\mbox{resp. } \left( a^{(i)} a^{(j - 1)}\right)^{\prime} = a^{(i + 1)}a^{(j - 1)} + a^{(i)}a^{(j)} = 0 \right)$$

        Thus, 
        $$ a^{(i + 1)}b^{(j - 1)} = -a^{(i)}b^{(j)} \neq 0 $$
        $$\left(\mbox{resp. } a^{(i + 1)}a^{(j - 1)} = -a^{(i)}a^{(j)} \neq 0 \right)$$
        and we obtained a contradiction with minimality of $j$.
        So, $j = 0$.
	\end{proof}

    Consider a differential $k$-algebra $A$.
    A differential polynomial $p(x_1, \ldots, x_n) \in A\{ x_1, \ldots, x_n \}$ is called \textit{multilinear} if it is linear with respect to all variables.
    
    \begin{proposition}\label{prop:no_identity}
		Let $A$ be a prime differential algebra over a field $k$ of characteristic zero with nonzero derivation.
        Then, for any nonzero differential polynomial $p \in A\{ x_1, \ldots, x_n \}$ there are $a_1, \ldots, a_n \in A$ such that $p(a_1, \ldots, a_n) \neq 0$.
	\end{proposition}
    
    \begin{proof}
		Since $\Char k = 0$, the polynomial $p$ can be assumed to by multilinear (see \cite[Th. 1.3.8]{zaicev}).
        By $T(m, n)$ we denote the statement 
        
        $T(m, n) = $ \textit{"For any multilinear polynomial $p$ in variables $x_1, \ldots, x_m$ such that $p$ does not depend on $x_i^{(s)}$ if $s > n$, there exist $a_1, \ldots, a_m \in A$ such that $p(a_1, \ldots, a_m) \neq 0$".}
        
        Below we establish several properties of these statements.
        
		\begin{lemma}\label{lem:Tmn}
        \begin{enumerate}
			\item $T(1, 0)$ and $T(1, 1)$ are true.
            
            \item For all $m \geqslant 2$ and $n \geqslant 0$ statements $T(m - 1, n)$ and $T(1, n)$ imply $T(m, n)$.
            
            \item Assume that $n > 1$. If $T(m , n - 1)$ is true for all $m$, then $T(1, n)$ is true.
        \end{enumerate}
		\end{lemma}

		\begin{proof}
        \begin{enumerate}
			\item We need to consider a polynomial of the form $p(x) = ax^{\prime} + bx$.
	        Assume that $p(x) = 0$ for all $x \in A$.
	        If $a = 0$ it implies that $bx = 0$ for all $x \in A$.
            But due to Proposition~\ref{prop:prime_characterization} for any nonzero $x \in A$ there exists $n$ such that $bx^{(n)} \neq 0$.
        	If $a \neq 0$, for all $x, y \in A$ we have 
	        $$a(xy)^{\prime} + bxy = ax^{\prime}y + axy^{\prime} + bxy = 0 \mbox{ and } ax^{\prime} + bx = 0$$
	        These equations imply $axy^{\prime} = 0$.
            Choose nonzero $a_1, a_2 \in A$ such that $a_2^{\prime} \neq 0$ and $aa_1 \neq 0$.
	        Due to Proposition~\ref{prop:prime_characterization} there exists $n$ such that $aa_1\left(a_2^{(n)}\right)^{\prime} \neq 0$.
            
			\item Consider $p(x_1, \ldots, x_m)$ as a polynomial in $x_1$ over $A\{ x_2, \ldots, x_m \}$.
        	The statement $T(m - 1, n)$ implies that there exist $a_2, \ldots, a_m \in A$ such that $p(x_1, a_2, \ldots, a_m)$ is a nonzero polynomial in $x_1$ over $A$.
	        Then, $T(1, n)$ implies that there exists $a \in A$ such that $p(a, a_2, \ldots, a_m) \neq 0$.

			\item Let $p(x) = c_n x^{(n)} + \ldots + c_0 x$ be a differential polynomial such that $p(a) = 0$ for all $a \in A$.
        	The polynomial $q(x, y) = p(xy) - xp(y) - yp(x)$ is nonzero, because it involves the monomial $nc_n x^{(n - 1)}y^{\prime}$.
	        Moreover, $q(x, y)$ does not depend on $x^{(n)}$ and $y^{(n)}$.
	        Thus, $T(n - 1, 2)$ implies that there exist $a, b \in A$ such that $q(a, b) \neq 0$.
	        Hence, either $p(a) \neq 0$ or $p(b) \neq 0$.
		\end{enumerate}
		\end{proof}

		Let us return to the proof of Proposition~\ref{prop:no_identity}.
        We will prove that $T(m, n)$ is true for all $m \geqslant 1$ and $n \geqslant 0$ by induction on $n$.
        The base cases $n = 0$ and $n = 1$ follow from the first two statements of Lemma~\ref{lem:Tmn}.
        
        Assume that $T(m, n_0)$ is true for all $m \geqslant 1$ and $n_0 < n$.
        Then, the third statement of Lemma~\ref{lem:Tmn} implies that $T(1, n)$ is true.
        Now, we can prove $T(m, n)$ for all $m$ by induction on $m$ using the second statement of Lemma~\ref{lem:Tmn}. 
	\end{proof}
    
    \begin{remark}
		Proposition~\ref{prop:no_identity} is not true for semiprime algebras.
        Consider the algebra $A = k[a] \oplus k[b]$ with $a^{\prime} = 0$ and $b^{\prime} = 1$.
        Let $p(x) = ax^{\prime}$.
        Then, $p(x) = 0$ for all $x \in A$.
	\end{remark}

	Let us return to the proof of Theorem~\ref{th:diff_to_lie}.
    It is sufficient to consider principal ideals generated by $a\partial$ and $b\partial$.
	Let $A_{id}$ be the algebra $A$ with adjoined identity.
    Theorem \cite[3.1]{DeltaPrimeNot2} implies that $\Diff A_{id}$ is prime.
    Hence, there exists a Lie polynomial $L(x, y, z)$ over $\Diff A$, which is linear with respect to $x$ and $y$, such that $L(a\partial, b\partial, \partial) \neq 0$.
    There exists a nonzero $p(x, y, t) \in A \{ x, y, t \}$ such that $L(a\partial, b\partial, t\partial) = p(a, b, t)\partial$.
    Proposition~\ref{prop:no_identity} implies that there exists $c \in A$ such that $p(a, b, c) \neq 0$.
    So, $L(a\partial, b\partial, c\partial) \neq 0$.
    Hence, product of ideals generated by $a\partial$ and $b\partial$ in $\Diff A$ is nonzero.

\end{proof}

{\bf Acknowledegements. } The first author was supported by the Austrian Science Fund FWF grant Y464-N18.
Authors thank the referee for helpful comments and suggestions.

%%%%%%%%%%%%%%%%%%%%%%%%%%%%

\end{document}